\documentclass[a4paper]{amsart}
\usepackage{amsmath,amssymb,amsthm}
\usepackage[utf8]{inputenc}
\usepackage[T1]{fontenc}
\usepackage{libertine}
\usepackage[libertine,cmintegrals,cmbraces]{newtxmath}
\usepackage{scalerel}
\usepackage{multicol}
\usepackage{enumitem}
\usepackage{comment} 
\usepackage{microtype}
\usepackage{xcolor}
\definecolor{darkred}{RGB}{139,0,0}
\definecolor{darkblue}{RGB}{0,0,139}
\definecolor{darkgreen}{RGB}{0,100,0}
\usepackage[linktocpage]{hyperref}
\hypersetup{
  colorlinks   = true, 
  urlcolor     = darkred, 
  linkcolor    = darkred, 
  citecolor   = darkgreen}
\usepackage{tikz}
\usepackage{tikz-cd}
\usepackage[capitalise]{cleveref}

\usepackage{stmaryrd}

\setenumerate[1]{label=(\roman*),}

\setitemize[1]{label=\raisebox{0.25ex}{\tiny$\bullet$}}

\tikzset{
  symbol/.style={
    draw=none,
    every to/.append style={
      edge node={node [sloped, allow upside down, auto=false]{$#1$}}}}}

\AtBeginDocument{%
   \def\MR#1{}}

\newcommand\smallsquare{\mathbin{\text{\raise0.17ex\hbox{\scalebox{.7}{$\blacksquare$}}}}}

\newcommand{\id}{\ensuremath{\mathrm{id}}}

\newcommand{\BO}{\ensuremath{\mathrm{BO}}}

\newcommand{\BSO}{\ensuremath{\mathrm{BSO}}}

\newcommand{\Diff}{\ensuremath{\mathrm{Diff}}}

\newcommand{\BlockDiff}{\ensuremath{\widetilde{\mathrm{Diff}}}}

\newcommand{\BDiff}{\ensuremath{\mathrm{BDiff}}}

\newcommand{\BlockBDiff}{\ensuremath{\mathrm{B\widetilde{Diff}}}}

\newcommand{\hAut}{\ensuremath{\mathrm{hAut}}}

\newcommand{\BhAut}{\ensuremath{\mathrm{BhAut}}}

\newcommand{\catsingle}[1]{\ensuremath{\mathcal{#1}}}

\newcommand{\oG}{\ensuremath{\mathrm{G}}}

\newcommand{\oO}{\ensuremath{\mathrm{O}}}

\newcommand{\bfH}{\ensuremath{\mathbf{H}}}

\newcommand{\bfR}{\ensuremath{\mathbf{R}}}
\newcommand{\bfZ}{\ensuremath{\mathbf{Z}}}
\newcommand{\bfQ}{\ensuremath{\mathbf{Q}}}

\newcommand{\cN}{\ensuremath{\catsingle{N}}}

\newcommand{\cR}{\ensuremath{\catsingle{R}}}
\newcommand{\cS}{\ensuremath{\catsingle{S}}}

\newcommand{\ra}{\rightarrow}

\newcommand{\lra}{\longrightarrow}

\newcommand{\xra}[1]{\xrightarrow{#1}}
\newcommand{\xlra}[1]{\overset{#1}{\longrightarrow}}

\newcommand{\scal}{\mathrm{scal}}
\newcommand{\sect}{\mathrm{sec}}
\newcommand{\Ric}{\mathrm{Ric}}
\newcommand{\st}{\mathrm{st}}

\newcommand{\BG}{\mathrm{BG}}

\newtheorem*{nthm}{Theorem}
\newtheorem*{ncor}{Corollary}

\newtheorem*{nlem}{Lemma}

\theoremstyle{definition}

\theoremstyle{remark}

\newtheorem{rem}{Remark}

\newtheorem*{nrem}{Remark}

\begin{document}

\title{An $\text{HP}^2$-bundle over $\text{S}^4$ with nontrivial $\hat{\text{A}}$-genus}

%\subjclass[2010]{57R52, 19D50, 57R65, 55P47}

\author{Manuel Krannich}
\email{krannich@dpmms.cam.ac.uk}
\address{Centre for Mathematical Sciences, Wilberforce Road, Cambridge CB3 0WB, UK}

\author{Alexander Kupers}
\email{a.kupers@utoronto.ca}
\address{Department of Computer and Mathematical Sciences, University of Toronto Scarborough, 1265 Military Trail, Toronto, ON M1C 1A4, Canada}

\author{Oscar Randal-Williams}
\email{o.randal-williams@dpmms.cam.ac.uk}
\address{Centre for Mathematical Sciences, Wilberforce Road, Cambridge CB3 0WB, UK}

\begin{abstract}
We explain the existence of a smooth $\bfH P^2$-bundle over $S^4$ whose total space has nontrivial $\hat{A}$-genus. Combined with an argument going back to Hitchin, this answers a question of Schick and implies that the space of Riemannian metrics of positive sectional curvature on a closed manifold can have nontrivial higher rational homotopy groups.
%
%\vspace{1ex}
%
%Nous expliquons l’existence d’un fibré différentiel de base $S^4$ et fibre $\bfH P^2$, dont l’espace total est de $\hat{A}$-genre non-trivial. En combinant ce resultat avec un argument de Hitchin, ceci répond à une question de Schick et implique que l’espace de métriques riemanniennes de courbure sectionnelle positive sur une variété fermée peut avoir des groupes d’homotopie rationnelle supérieures non-triviaux.  
\end{abstract}

\maketitle

In view of applications to spaces of Riemannian metrics with positive curvature, there has been recent interest in constructing smooth fibre bundles over spheres whose total space has nontrivial $\hat{A}$-genus. In their work on the space of positive scalar curvature metrics, Hanke--Schick--Steimle \cite[Corollary\,1.6]{HankeSchickSteimle} showed that such bundles exist for every dimension of the base sphere. However, as noted on page 337 loc.\,cit., their method does not yield bundles with an explicit description of the fibre, though they are able to show that the fibre may be chosen to carry a metric of positive scalar curvature using a theorem of Stolz. For applications to spaces of metrics with positive sectional or Ricci curvature it is desirable to have  examples with fibre carrying such a metric. In \cite[p.\,3999]{BotvinnikEbertWraith} (see also \cite[Section\,9]{SchickICM}) it is said that this ``seems to be a very difficult problem'': we offer the following solution.

\begin{nthm}\label{thm:main} 
There exists a smooth oriented fibre bundle $\bfH P^2\ra E^{12}\ra S^4$ with $\hat{A}(E)\neq 0$.
\end{nthm}

\begin{nrem}
The argument we give also shows that this fibre bundle may be assumed to have a section with trivial normal bundle (see \cref{rem:section}), and provides analogous $\bfH P^n$-bundles over $S^4$ for all even $n \geq 2$ (see \cref{rem:higher-dimensions}). It can certainly be extended further.
\end{nrem}

The standard Riemannian metric $g_{\st}$ on $\bfH P^2$ has positive sectional curvature, so pulling back $g_{\st}$ along orientation-preserving diffeomorphisms yields a map
\[
(-)^*g_{\st}\colon \Diff(\bfH P^2)\lra \cR^{\sect>0}(\bfH P^2)\subset \cR^{\Ric>0}(\bfH P^2)\subset \cR^{\scal>0}(\bfH P^2)
\] 
from the group of diffeomorphisms of $\bfH P^2$ in the smooth topology to the spaces of Riemannian metrics on $\bfH P^2$ having positive sectional, Ricci, or scalar curvature. By an argument of Hitchin \cite{Hitchin} (see for instance \cite[p.\,3999]{BotvinnikEbertWraith} for an explanation of this), the theorem has the following corollary, which answers a question of Schick \cite[p.\,30]{OWL} and provides an example as asked for in \cite[Remark\,2.2]{BotvinnikEbertWraith}. 

\begin{ncor}
The induced map
\[\pi_3((-)^*g_{\st})\otimes\bfQ\colon \pi_3(\Diff(\bfH P^2);\id)\otimes\bfQ\lra  \pi_3(\cR ^{\scal>0}(\bfH P^2);g_{\st})\otimes\bfQ\]
is nontrivial, so in particular $\pi_3(\cR ^{\sect>0}(\bfH P^2);g_{\st})\otimes\bfQ\neq0$ and $\pi_3(\cR ^{\Ric>0}(\bfH P^2);g_{\st})\otimes\bfQ\neq0$.
\end{ncor}

\subsection*{Proof of the Theorem}

Smooth $\bfH P^2$-bundles over $S^4$ (together with an identification of the fibre over the basepoint with $\bfH P^2$) are classified by $\pi_4(\BDiff(\bfH P^2))$, so our task is to show that the morphism $\hat{A}\colon \pi_4(\BDiff(\bfH P^2))\ra \bfQ$ assigning an $\bfH P^2$-bundle $E\ra S^4$ the $\hat{A}$-genus of the total space is nontrivial. This morphism admits a factorisation of the form
\begin{equation}\label{equ:factorsation}
\pi_4(\BDiff(\bfH P^2))\lra \pi_4(\BlockBDiff(\bfH P^2))\xlra{\hat{A}}\bfQ,
\end{equation}
where $\BlockDiff(\bfH P^2)$ is the block diffeomorphism group of $\bfH P^2$ and the first map is induced by the canonical comparison map $\Diff(\bfH P^2)\ra \BlockDiff(\bfH P^2)$. This factorisation follows for instance from \cite[Theorem 1]{EbertRW}, but there is also a more direct argument: via the canonical isomorphism $\pi_4(\BDiff(\bfH P^2))\cong \pi_3(\Diff(\bfH P^2);\id )$, the morphism $\hat{A}\colon \pi_4(\BDiff(\bfH P^2)\ra \bfQ$ is given by mapping a diffeomorphism $\phi\colon D^3\times \bfH P^2\ra D^3\times \bfH P^2$ that is the identity on the boundary and commutes with the projection to $S^3$ to the $\hat{A}$-genus of the glued manifold $D^4\times \bfH P^2\cup_{\phi\cup\id} D^4\times \bfH P^2$. This description of the morphism makes clear that it factors through the map $\pi_3(\Diff(\bfH P^2);\id )\ra \pi_0(\Diff_{\partial}(D^3\times\bfH P^2))\cong \pi_3(\BlockDiff(\bfH P^2);\id)$ that only remembers the underlying isotopy class of $\phi$.

We thus have to show nontriviality of the composition \eqref{equ:factorsation}. It suffices to check this after rationalisation, which makes the first map surjective:

\begin{nlem}\label{lem:morlet}
The map $\pi_4(\BDiff(\bfH P^2)) \otimes\bfQ\lra \pi_4(\BlockBDiff(\bfH P^2))\otimes\bfQ$ surjective.
\end{nlem}

\begin{proof}
Choosing an embedded disc $D^{8}\subset \bfH P^2$, we consider the commutative square
\begin{center}
\begin{tikzcd}
\BDiff_\partial(D^{8})\dar\rar&\BDiff(\bfH P^2)\dar\\
\BlockBDiff_\partial(D^{8})\rar&\BlockBDiff(\bfH P^2)
\end{tikzcd}
\end{center}
whose horizontal maps are induced by extending (block) diffeomorphisms of $D^{8}$ that are the identity on the boundary to $\bfH P^2$ by the identity. The claim follows by showing that the third rational homotopy group of the right vertical homotopy fibre vanishes for which we note that, since $\bfH P^2$ is $3$-connected, the square is $4$-cartesian by Morlet's lemma of disjunction \cite[Corollary 3.2, p.\,29]{BLR}, so it suffices to show that the third rational homotopy group of the left vertical map is trivial. Since $\pi_i(\BlockBDiff_\partial(D^{2n}))\cong\pi_0\Diff_\partial(D^{{2n}+i-1})\cong \Theta_{2n+i}$  vanishes rationally as the group $\Theta_{2n+i}$ of homotopy $(2n+i)$-spheres is finite, the claim follows from $\pi_3(\BDiff_\partial(D^{8}))\otimes\bfQ=0$ which holds by \cite[Theorem 4.1]{RWUpperBound}.
\end{proof}

Given the lemma, we are left to show that the map $\hat{A}\colon \pi_4(\BlockBDiff(\bfH P^2))\ra\bfQ$ is nontrivial, which we shall do after precomposition with the map
\begin{equation*}%\label{eq:HitFRomBehind}
\pi_4(\hAut(\bfH P^2)/\BlockDiff(\bfH P^2);\id)\lra \pi_4(\BlockBDiff(\bfH P^2))
\end{equation*}
induced by the inclusion of the homotopy fibre of the comparison map $\BlockBDiff(\bfH P^2)\ra \BhAut(\bfH P^2)$, where $\hAut(\bfH P^2)$ is the topological monoid of homotopy automorphisms of $\bfH P^2$. Considering this homotopy fibre is advantageous since the $h$-cobordism theorem provides an isomorphism 
\[
\pi_4(\hAut(\bfH P^2)/\BlockDiff(\bfH P^2))\cong \cS_\partial(D^4\times \bfH P^2)
\]
to the \emph{structure group} of $D^4\times \bfH P^2$ relative to $\partial D^4\times \bfH P^2$ in the sense of surgery theory (see \cite{WallBook} for background on surgery theory, especially Chapter 10), which in turn fits into the \emph{surgery exact sequence} of abelian groups
\[
0=L_{13}(\bfZ) \xlra{\partial} \cS_\partial(D^4\times \bfH P^2) \xlra{\eta} \cN_\partial(D^4\times \bfH P^2)\xlra{\sigma} L_{12}(\bfZ)\cong\bfZ
\] 
featuring the surgery obstruction map $\sigma$ from the normal invariants $\cN_\partial(D^4\times \bfH P^2)$ to the $L$-group $L_{12}(\bfZ)\cong \bfZ$. The standard smooth structure on $D^4\times \bfH P^2$ provides an isomorphism 
\[
\cN_\partial(D^4\times \bfH P^2)\cong [S^4 \wedge \bfH P^2_+,\oG/\oO],
\]
where $[-,-]$ stands for based homotopy classes and $\oG/\oO$ is the homotopy fibre of the map $\BO\ra \BG$ classifying the underlying stable spherical fibration of a stable vector bundle.

As $\BG$ has trivial rational homotopy groups, the map 
\[
[S^4 \wedge \bfH P^2_+,\oG/\oO] \lra [S^4 \wedge \bfH P^2_+,\BO] = \widetilde{KO}^0(S^4 \wedge \bfH P^2_+)
\]
 is rationally an isomorphism. Furthermore the Pontrjagin character gives an isomorphism
\[
\textstyle{\mathrm{ph}(-) = \mathrm{ch}(- \otimes_\bfR \mathbf{C})\colon \widetilde{KO}^0(S^4 \wedge \bfH P^2_+)\otimes \bfQ \overset{\cong}\lra \bigoplus_{i \geq 0} \widetilde{H}^{4i}(S^4 \wedge \bfH P^2_+;\bfQ) = u \cdot \bfQ[z]/(z^3)},
\]
where $u \in \widetilde{H}^4(S^4;\bfQ)$ denotes the cohomological fundamental class, and $z \in H^4(\bfH P^2; \bfQ)$ is the usual generator. Therefore for any triple $(A, B, C) \in \bfQ^3$ there exists a nonzero $\lambda \in \bfZ$ and a normal invariant $n \in \cN_\partial(D^4\times \bfH P^2)$ whose underlying stable vector bundle $\xi$ has $\mathrm{ph}(\xi) = \lambda \cdot u \cdot  (A + Bz + Cz^2)$. Since $S^4\wedge \bfH P^2_+$ has no nontrivial cup-products among elements of positive degree, we have $\mathrm{ph}_i(\xi)=(-1)^{i+1}/{(2i-1)!}\cdot p_i(\xi)$ and hence
\begin{equation}\label{eq:PontClasses}
p_1(\xi) = \lambda A \cdot u \quad\quad\quad p_2(\xi) = -6 \lambda B \cdot u \cdot z\quad\quad\quad p_3(\xi) = 120 \lambda C \cdot u \cdot z^2.
\end{equation}

To evaluate the surgery obstruction map $\sigma$, recall that a normal invariant $n$ with underlying stable vector bundle $\xi$ is represented by a degree 1 normal map
\begin{equation}\label{eq:Def1Normal}
\begin{tikzcd}
\nu_M \dar \rar{\hat{f}}& \nu_{D^4 \times \bfH P^2} \oplus\xi \dar\\
M^{12} \rar{f}& D^4 \times \bfH P^2,
\end{tikzcd}
\end{equation}
where $\partial M = \partial D^4 \times \bfH P^2$ and $f$ and $\hat{f}$ restrict to the identity maps on the boundary. Here $\nu_{(-)}$ denotes the stable normal bundle of a manifold. The surgery obstruction is unchanged by gluing into $M$ and $D^4 \times \bfH P^2$ a copy of $D^4 \times \bfH P^2$ along the identification of their boundaries with $\partial D^4 \times \bfH P^2$, and extending $f$ and $\hat{f}$ trivially, giving rise to a degree 1 normal map to $f'\colon M' \to S^4 \times \bfH P^2$. The surgery obstruction may then be expressed in terms of the signatures of these manifolds, as
\[\sigma(n) = \tfrac{1}{8}\big(\mathrm{sign}(M') - \mathrm{sign}(S^4 \times \bfH P^2)\big).\]

The signature of $S^4 \times \bfH P^2$ is trivial, and that of $M'$ may be computed in terms of the Hirzebruch signature theorem as the evaluation $\int_{M'} L(TM')$ of the $L$-class. As $f'$ has degree 1 and pulls back $\nu_{S^4 \times \bfH P^2} \oplus\xi$ to the stable inverse of $TM'$, we have
\begin{equation}\label{equ:sgn-M'}\textstyle{\mathrm{sign}(M') = \int_{M'} L(TM') = \int_{S^4  \times \bfH P^2} L(TS^4)\cdot L(T\bfH P^2)\cdot L(-\xi)}.\end{equation}
The first terms of the total $L$-class are given as
\begin{align*}L &= 1 + \tfrac{p_1}{3} + \tfrac{7\cdot p_2 - p_1^2}{45} + \tfrac{62 \cdot p_3 - 13 \cdot p_1p_2+2 \cdot p_1^3}{945} + \cdots \label{eq:Lclass},
\end{align*}
which we combine with $p(T\bfH P^2) = 1 + 2z + 7z^2$ from \cite[Satz 1]{Hirzebruch} to compute
\begin{align*}
L(TS^4)&=1\\ L(T\bfH P^2)&=1 + \tfrac{2}{3}\cdot z + z^2\\ L(-\xi)&=1+\lambda(-\tfrac{1}{3}A\cdot u + \tfrac{14 }{15}B\cdot (u\cdot z) - \tfrac{496}{63}C \cdot (u\cdot z^2))
\end{align*}
and thus
\[8 \sigma(n) = \mathrm{sign}(M') = \lambda(-\tfrac{1}{3}A +\tfrac{28}{45}B-\tfrac{496}{63}C).\]

It follows that for each triple $(A, B, C) \in \bfQ^3$ satisfying $\tfrac{1}{3}A -\tfrac{28}{45}B+\tfrac{496}{63}C=0$ there exists a non-zero $\lambda \in \bfZ$ and a degree 1 normal map as in \eqref{eq:Def1Normal} with $f$ a homotopy equivalence and with $\xi$ having Pontrjagin classes as in \eqref{eq:PontClasses}. This gives a smooth block $\bfH P^2$-bundle structure on the composition 
\[
M' \xra{f'} S^4 \times \bfH P^2 \xra{\pi_1} S^4
\]
giving rise to a class in $\pi_4(\BlockBDiff(\bfH P^2)$, so it remains to evaluate $\hat{A}(M')$. As in \eqref{equ:sgn-M'}, we get
\[\textstyle{\hat{A}(M') = \int_{M'} \hat{A}(TM') = \int_{S^4 \times \bfH P^2} \hat{A}(TS^4) \cdot \hat{A}(T\bfH P^2) \cdot \hat{A}(-\xi)},\]
which we combine with the formula for the first terms of the total $\hat{A}$-class
\begin{align*}
\hat{A} &= 1 - \tfrac{p_1}{24} + \tfrac{-4\cdot p_2+7\cdot p_1^2}{5760} + \tfrac{-16 \cdot p_3 + 44 \cdot p_2p_1 -31 \cdot p_1^3}{967680} + \cdots
\end{align*}
to compute
\begin{align*}
\hat{A}(TS^4)&=1\\ \hat{A}(T\bfH P^2) &= 1 - \tfrac{1}{12}\cdot z \\ \hat{A}(-\xi) &= 1+\lambda(\tfrac{1}{24}A\cdot u - \tfrac{1}{240}B\cdot (u\cdot z) + \tfrac{1}{504}C\cdot (u\cdot z^2))
\end{align*}
from which we conclude
\[\hat{A}(M') = \lambda(\tfrac{1}{2880}B + \tfrac{1}{504}C).\]
As there are clearly triples $(A,B,C) \in \bfQ^3$ satisfying 
\[-\tfrac{1}{3}A +\tfrac{28}{45}B-\tfrac{496}{63}C=0\quad\text{and}\quad \tfrac{1}{2880}B + \tfrac{1}{504}C \neq 0,\]
 this finishes the argument.

\begin{rem}
A fibre bundle $\pi \colon E \to S^4$ constructed in this way is fibre homotopy equivalent to the trivial bundle $\pi_1 \colon S^4 \times \bfH P^2 \to S^4$, and under this fibre homotopy trivialisation we have $p_1(TE) = 2\cdot (1 \otimes z) - \lambda A \cdot (u \otimes 1)$. Thus $p_1(TE)^3 = - 12 \lambda A \cdot (u \otimes z^2)$, and so 
\[\textstyle{\int_E p_1(TE)^3 = -12 \lambda A\quad\text{and} \quad\int_E \hat{A}_3(TE) = \lambda(\tfrac{1}{2880}B + \tfrac{1}{504}C).}\] This argument therefore guarantees the existence of a 2-dimensional subspace of the group $\pi_4(B\Diff(\bfH P^2))\otimes\bfQ$, detected by the characteristic numbers $\int_E p_1(T E)^3$ and $\int_E \hat{A}_3(T E)$. 
\end{rem}

\begin{rem}\label{rem:section}
The Theorem can be slightly strengthened: we may in addition assume that the smooth $\bfH P^2$-bundle $\pi\colon E^{12}\ra S^4$ admits a smooth section with trivial normal bundle, which may be helpful for fibrewise surgery constructions. 

To see this, note that a fibre bundle $\pi\colon E \to S^4$ as constructed above is fibre homotopy equivalent to the trivial bundle $\pi_1 \colon S^4 \times \bfH P^2 \to S^4$, so it admits a smooth section $s\colon S^4 \to E$ corresponding to a trivial section of the trivial bundle. By the description of $p_1(TE)$ in the previous remark we have $s^* p_1(TE) = -\lambda A \cdot u$. If we choose $(A=0, B= \tfrac{496}{63}, C=\tfrac{28}{45})$, which is a triple whose surgery obstruction vanishes, then the corresponding bundle has $s^* p_1(TE)=0$ and $\hat{A}(E) \neq 0$. As $TS^4$ is stably trivial, it follows that the normal bundle of $s(S^4) \subset E$ has trivial first Pontrjagin class. This implies that the normal bundle is trivial, since $p_1\colon \pi_4(\BSO(8)) \to \bfZ$ is injective as $\pi_4(\BSO(8))\cong\pi_4(\BSO)$ is stable.
\end{rem}

\begin{rem}\label{rem:higher-dimensions}
	The argument can be generalised to prove that for any even $n \geq 2$ there is a smooth $\bfH P^n$-bundle $E^{4n+4} \to S^4$ with $\hat{A}(E) \neq 0$, so the Corollary holds for $\bfH P^n$ as well. 
	
	Indeed, the application of Morlet's lemma only required the fibre to be at least $8$-dimensional and $3$-connected. 
	Similar to the above, one argues that for any pair $(A,C)\in\bfQ^2$ there exists a nonzero $\lambda\in\bfZ$ and a normal invariant $n\in\cN_\partial(D^4\times \bfH P^n)$ with underlying stable vector bundle $\xi$ such that 
	\[p_1(\xi) = \lambda A\cdot u\quad\quad p_i(\xi) = 0 \text{ for $1 < i < n$,}\quad\quad p_{n+1}(\xi) = \lambda(2n+1)!(-1)^n C \cdot u \cdot z^n.\]
	Using that the coefficient of $z^n$ in $L(T\bfH P^n)$ is $1$ by the Hirzebruch's signature theorem, we see that the surgery obstruction satisfies $8\sigma(n) =\lambda(-\tfrac{1}{3}A+h_{n+1}(2n+1)!(-1)^{n+1} C)$ where $h_n$ is the coefficient of $p_n$ in the total $L$-class. In particular, we can always find pairs $(A,C)$ with $C \neq 0$ and $8\sigma(n) = 0$. As the coefficient of $z^n$ in $\hat{A}(T\bfH P^n)$ vanishes, because $\bfH P^n$ admits a metric of positive scalar curvature, it holds $\hat{A}(E) = \lambda a_{n+1}(2n+1)!(-1)^{n+1} C\neq0$ with $a_n$ the coefficient of $p_n$ in the total $\hat{A}$-class, which is easily proved to be non-zero using \cite[Ch.~1.\S1 (10)]{HirzebruchBook}.
\end{rem}

\begin{rem}\label{rem:high-base-dimension}
If one is willing to instead consider $\bfH P^n$-bundles over $S^{4m}$ for $n$ sufficiently large compared with $m$, then one may replace the appeal to the Lemma by the more classical \cite[Corollary D]{BurgheleaLashof}, which implies that the map
\[\pi_{4m}(\hAut(\bfH P^n)/\Diff(\bfH P^n);\id) \otimes \bfZ[\tfrac{1}{2}] \lra \pi_{4m}(\hAut(\bfH P^n)/\BlockDiff(\bfH P^n);\id) \otimes \bfZ[\tfrac{1}{2}]\]
is (split) surjective as long as $4m-1$ lies in the pseudoisotopy stable range for $\bfH P^n$ (so $3m < n$ suffices, by \cite{Igusa}). See also \cite[Theorem 1]{Burghelea}. One must still produce an appropriate element of $\cS_\partial(D^{4m}\times \bfH P^n)$, which may be approached as in Remark \ref{rem:higher-dimensions}.
\end{rem}

\subsection*{Acknowledgements}
We thank Thomas Schick for encouraging us to write this note, and Johannes Ebert for a comment which reminded us of \cite{Burghelea} and \cite{BurgheleaLashof}. The first and third author were supported by a Philip Leverhulme Prize from the Leverhulme Trust. The third author was also supported by the ERC under the European Union's Horizon 2020 research and innovation programme (grant agreement No.\ 756444). 

\bibliographystyle{amsalpha}
\bibliography{literature}
\vspace{0.2cm}

\end{document}